\documentclass{amsart} 
\makeatletter
\@namedef{subjclassname@2020}{%
  \textup{2020} Mathematics Subject Classification}
\makeatother
\usepackage{latexsym}
\usepackage{amssymb}
\usepackage{graphicx}
\usepackage{wrapfig}
\usepackage{amsthm}
\usepackage{float}
\usepackage{epsfig}
\usepackage{multirow}
\usepackage{xcolor}
\usepackage{caption}
\usepackage[toc,page]{appendix}
\usepackage{old-arrows} 
\usepackage{caption}
\usepackage{subcaption}
\usepackage{lmodern, mathtools}
\usepackage{mathtools}
\usepackage{hyperref}
\RequirePackage{color}
\usepackage{tikz}
\usepackage{amscd,enumerate}
\usepackage{amsfonts}
\hypersetup{ colorlinks,
linkcolor=blue,
filecolor=green,
urlcolor=blue,
citecolor=blue }
\usepackage{helvet}
\usepackage{bigstrut}
\usepackage{float, caption, booktabs, cellspace}
\usepackage{calrsfs}
\DeclareMathAlphabet{\pazocal}{OMS}{zplm}{m}{n}

\makeatletter
\@addtoreset{subfigure}{framenumber}
\makeatother

\usepackage{etoolbox}
\AtBeginEnvironment{figure}{\setcounter{subfigure}{0}}

\newtheorem{lemma}{Lemma}[section]
\newtheorem{corollary}[lemma]{Corollary}
\newtheorem{theorem}[lemma]{Theorem}
\newtheorem{proposition}[lemma]{Proposition}
\newtheorem{remark}[lemma]{Remark}
\newtheorem{definition}[lemma]{Definition}

\newtheorem{example}[lemma]{Example}

\input xygraph
\input xy
\xyoption{all}

\usepackage{array}
\newcolumntype{L}[1]{>{\raggedright\let\newline\\\arraybackslash\hspace{0pt}}m{#1}}
\newcolumntype{C}[1]{>{\centering\let\newline\\\arraybackslash\hspace{0pt}}m{#1}}
\newcolumntype{R}[1]{>{\raggedleft\let\newline\\\arraybackslash\hspace{0pt}}m{#1}}

\usepackage{multicol}

\newcommand{\bK}{\mathbb{K}}
\def\l{\lambda}

\def\remove#1{}
\def\H{\mathcal H}
\def\Id{\mathcal I}

\def\span#1{\text{\rm span}(#1)}
\newcommand{\ann}{{\rm ann}}
\newcommand{\Supp}{{\rm supp}}

\newcommand{\N}{{\mathbb{N}}}

\newcommand{\K}{{\mathbb{K}}}

\newcommand{\h}{\mathcal{H}}

\title[Connecting ideals in evolution algebras with hereditary subsets]{Connecting ideals in evolution algebras with hereditary subsets  of its associated graph}

\author[Y. Cabrera]{Yolanda Cabrera Casado}

\author[D. Mart\'{\i}n]{Dolores Mart\'{\i}n Barquero}
\author[C. Mart\'{\i}n]{C\'andido Mart\'{\i}n Gonz\'alez}
\author[A. Tocino]{Alicia Tocino}

\address{Departamento de Matem\'atica Aplicada, E.T.S. Ingenier\'\i a Inform\'atica, Universidad de M\'alaga, Campus de Teatinos s/n. 29071 M\'alaga.   Spain. }
\email{yolandacc@uma.es}
\address{Departamento de Matem\'atica Aplicada, Escuela de Ingenier\'\i as Industriales, Universidad de M\'alaga, Campus de Teatinos s/n. 29071 M\'alaga.   Spain.}
\email{dmartin@uma.es}

\address{Departamento de \'Algebra Geometr\'{\i}a y Topolog\'{\i}a, Fa\-cultad de Ciencias, Universidad de M\'alaga, Campus de Teatinos s/n. 29071 M\'alaga. Spain.} \email{candido\_m@uma.es}

\address{Departamento de Matem\'atica Aplicada, E.T.S. Ingenier\'\i a Inform\'atica, Universidad de M\'alaga, Campus de Teatinos s/n. 29071 M\'alaga.   Spain.}
 \email{alicia.tocino@uma.es}

\thanks{ The  authors are supported by the Spanish Ministerio de Ciencia e Innovaci\'on through project  PID2019-104236GB-I00/AEI/10.13039/501100011033 and  by the Junta de Andaluc\'{\i}a  through project  FQM-336,  all of them with FEDER funds. 
}

\begin{document}

\begin{abstract}
In this article, we introduce a relation including  ideals of an evolution algebra  and hereditary subsets of vertices of its associated graph and establish some properties among them. This relation allows us to determine maximal ideals and ideals having the absorption property of an evolution algebra in terms of its associated graph. We also define a couple of order-preserving maps, one from the sets of ideals of an evolution algebra to that of hereditary subsets of the corresponding graph, and the other in the reverse direction. Conveniently restricted to the set of absorption ideals and to the set of hereditary saturated subsets, this is a monotone Galois connection.  According to the graph, we characterize  arbitrary dimensional finitely-generated (as algebras) evolution algebras under certain restrictions of its graph. Furthermore, the simplicity of finitely-generated perfect evolution algebras is described on the basis of the simplicity of the graph.
\end{abstract}

\maketitle

\section{Introduction }

In the literature of graph $C^*$-algebras, Leavitt path algebras, higher rank graph algebras, etc., we can find how certain structural properties of the algebra (simplicity, primeness, primitivity, etc.,)  can be characterized in terms of properties of its associated graph (see, for example, \cite{AbramsAranda} for simplicity, \cite{ArandaPardoSiles2} for primeness and \cite{ArandaPardoSiles} for primitivity  in Leavitt path algebras) to whose pioneer discoverers we recognize here. In fact, this philosophy is followed in a good number of publications about evolution algebras. 
In \cite{EL1}, where the way of associating a graph to an evolution algebra appeared for the first time, which we are going to use, the authors study the nilpotency of an evolution algebra regarding whether there exists oriented cycles in the associated graph. They also relate the indecomposability of an evolution algebra with  the connectivity of its graph. The space of derivations of some evolution algebras has also been analyzed through properties of its graph (see \cite{Elduque/Labra/2019}, \cite{PMP3}, \cite{YPMP}, \cite{Qaralleh/Mukhamedov/2019}). Recently, in \cite{Hilbert}, the authors  give conditions on the graph to ensure the existence of isomorphisms between Hilbert evolution algebras.

In our work, we associate a certain hereditary set of  vertices  of the associated graph to an ideal of the evolution algebra and vice versa. Under mild restrictions, this association is a Galois connection. 
This way of linking ideals with subsets of vertices allows us to detect on the graph whether the corresponding algebra contains ideals that verify certain properties and which are those ideals.

The article is organized as follows. In Section \ref{sec2}, we recall some basic concepts and properties  on directed graphs and evolution algebras theory that will be drawn on throughout this paper.  In Section \ref{sec3}, we start by defining, given an ideal $I$ of an evolution algebra over a field $\K$ with a natural basis $B=\{e_i\}_{i\in \Lambda}$, the hereditary subset $H_I$ of the associated graph, as well as, fixed a hereditary subset of the vertices $H$, we define  a subspace $I_H$ that turns out to be a basic ideal (Definition \ref{defHI}). These are the key objects that will be used along this work. Several properties and relations among them are shown in Proposition \ref{sabado1}. Moreover, we characterize the saturated hereditary subsets $H$, through the ideals $I_H$ of the algebra. 
In  Proposition \ref{sabado2}, we describe  maximal ideals $I$ that contains $A^2$. Furthermore, we also identify  maximal ideals $I$ such that  $A^2 \not\subset I$ in terms of maximal hereditary subsets  of vertices of its associated graph.

Consequently, in Corollary \ref{blanco}, we prove that 

for arbitrary dimensional perfect evolution algebras, all the maximal ideals are perfectly identified in the associated  graph and, in addition,  every maximal ideal is basic. As a summary, in Remark \ref{remark13} we explain  that  modulo some inevitable subspaces of codimension 1 of an arbitrary dimensional evolution algebra $A$ containing $A^2$, the maximal ideals are those associated to maximal hereditary  subsets of vertices.

To guarantee the existence of maximal ideals in an evolution algebra requires that the evolution algebra be finitely-generated (as algebra). So, we dedicate Section \ref{sec5} to study  the necessary conditions for an evolution algebra to be finitely-generated. In Lemma \ref{azul}, we state that if the set of vertices of the graph can not be viewed as the hereditary closure of a finite set, then the evolution algebra can not be finitely-generated.
Moreover, in the case in which the graph has a finite number of bifurcations, the finitely-generated evolution algebras are characterized (Proposition \ref{noche}).

In Section \ref{sec4},  we identify, for non-degenerate evolution $\K$-algebras, the ideals $I$ having the absorption property as those of the form $I=I_H$ with $H$ hereditary and saturated subset (Proposition \ref{sabado3} and Proposition \ref{propnueva}).
In Theorem \ref{carrillada}, we give  equivalences for the ideals of an evolution algebra having the absorption property. Together with Proposition \ref{sabado3}, it  leads us a way of finding these ideals by just searching the hereditary and saturated sets of the associated graph for non-degenerate evolution algebras.
For finite-dimensional perfect evolution algebras and an arbitrary ideal $I$, we obtain  that  $I$ has the absorption property and it is generated by the set $H_I$ (Theorem \ref{fiesta}). Moreover,  $I$ is a basic ideal (Corollary \ref{propbasic}). In the second part of this section, we show that in a finite-dimensional perfect evolution algebra, the set of all hereditary and saturated subsets of vertices and the set of all ideals having the absorption property with certain maps defined in Definition \ref{definicionapl} form a monotone Galois connection (Theorem \ref{Galois}).

In Section \ref{sec6}, we introduce the definition of 
simple directed  graph (Definition \ref{defgrafosimple})
and we prove that, for finitely-generated perfect evolution algebras, simplicity in the algebra is equivalent to simplicity in the associated graph (Proposition \ref{propsimple}).
In item (1) of Theorem \ref{teo64}, we show that, given a directed graph $E$, a hereditary subset  $H \subset E^0$ is maximal if and only if the only hereditary subsets of the quotient graph $E/H$ are $(E/H)^0$ and $\emptyset$.
In item (2) of Theorem \ref{teo64}, we prove that  simple graphs are  either an isolated vertex or a graph with no sources or sinks. Finally, in item (3) of Theorem \ref{teo64}, we show that the associated graph of a quotient evolution algebra under an ideal of the form $I_H$ is the quotient graph under the hereditary subset, $H$.

The examples distributed throughout the article, in addition to providing us with a large number of examples of evolution algebras of different types, show the advantage of using the graph to locate ideals with certain properties in the evolution algebra.

\section{Basic concepts}\label{sec2}

Althrough this work, the set of natural numbers will be denoted by $\N=\{0,1,\ldots\}$ and the notation $\N^*$ will stand for $\N^*:=\N\setminus\{0\}$. 

A \emph{directed graph} is a $4$-tuple $E=(E^0, E^1, r_E, s_E)$ 
consisting of two disjoint sets $E^0$, $E^1$ and two maps
$r_E, s_E: E^1 \to E^0$. All the graphs, in this paper, will be directed graphs so, when we say graph we mean directed graph. The elements of $E^0$ are called \emph{vertices} and the elements of 
$E^1$ are called \emph{edges} of $E$. Observe that $E^1$ and $E^0$ can have any cardinality. Further, for $f\in E^1$, $r_E(f)$ and $s_E(f)$ are 
 the \emph{range} and the \emph{source} of $f$, respectively.  
If there is no confusion with respect to the graph we are considering, we simply write $r(f)$ and $s(f)$. A
vertex $v$ for which $s^{-1}(v)=\emptyset$  is called a \emph{sink}, while a vertex $v$ for which $r^{-1}(v)=\emptyset$ is called a \emph{source}. 
A vertex $u\in E^0$ is said to be \emph{regular} if $0<\vert s^{-1}(u)\vert < \infty$. A \emph{path} $\mu$ of length $m$ is a finite chain of edges $\mu=f_1\ldots f_m$ such that $r(f_i)=s(f_{i+1})$ for $i=1,\ldots,m-1$.  We denote by $s(\mu):=s(f_1)$ the source of $\mu$ and $r(\mu):=r(f_m)$ the range of $\mu$.  We write $\mu^0$ the set of vertices of $\mu$. 
If $r(\mu)=s(\mu)$,
then $\mu$ is called a \emph{closed path}. 
The vertices will be considered trivial paths, that is to say, paths of length zero.
We define $\text{Path}(E)$ as the set of all paths in $E$. 
If $E$ is a directed graph and $S\subset E^0$, then denote by $T(S)$ the \emph{tree} of $S$ where $$T(S)=\{v\in E^0 \colon \text{ exist }\lambda \in \text{Path}(E) \text{ and } u \in S \text{ with } s(\lambda)=u, r(\lambda)=v\}.$$

A graph $E$ is said to be {\em row-finite} if for any $u\in E^0$ we have 
that $\vert s^{-1}(u)\vert$ is finite. A graph is said to satisfy \emph{Condition (Sing)} if among two vertices there is at most one edge.  
A directed graph is \textit{strongly connected} if given two different vertices there exists a path that goes from the first to the second one.

A vertex $u\in E^0$ is called a \textit{bifurcation vertex} (or it is said that there is a \textit{bifurcation at $u$}) if $\vert s^{-1}(u)\vert \geq 2$.

A subset $H\subset E^0$ is \textit{hereditary} if for all $u\in H$, $T(u)\subset H$. We will denote the subset of all subsets of $E^0$ that are hereditary  by $\h_E$, if there is no risk of confusion we will write $\h$. The \textit{ hereditary closure} of $S\subset E^0$ is the smallest hereditary subset containing $S$. It will be denoted by $\overline{S}$. We say that $H\in \H$  with $H\neq E^0$ is \textit{maximal} if $H\subset H'$, $H'\in \H$, implies $H=H'$ or $H'=E^0$. A set $H\subset E^0$ is \textit{saturated} if for any regular  vertex $u\in E^0$  we have that $r(s^{-1}(u)) \subset H$ implies $u\in H$.

An \emph{evolution algebra} over a field $\bK$ is a $\bK$-algebra $A$ which has a basis $B=\{e_i\}_{i\in \Lambda}$ such that $e_ie_j=0$ for every $i, j \in \Lambda$ with $i\neq j$. Such a basis is called a \emph{natural basis}.

Let $A$ be an evolution $\K$-algebra with a natural basis $B=\{e_i\}_{i\in \Lambda}$.
Denote by $M_B=(\omega_{ij})$ the \emph{structure matrix} of $A$ relative to $B$, i.e., $e_i^2 = \sum_{j\in \Lambda} \omega_{ji}e_j$. If $\ann(A):=\{x \colon x A =0 \}=\span{\{e_i \colon e_i^2=0\}}=0$ we say that $A$ is \textit{non-degenerate}. Otherwise we say that $A$ is \textit{degenerate}. In the particular case of $A^2=A$, we say that the evolution algebra is \emph{perfect}. For finite-dimensional evolution algebras, this is equivalent to $\det(M_B)\neq 0$. Let $u=\sum_{i\in \Lambda}\alpha_ie_i$ be an element of $A$. Recall that 
the {\emph{support of}} $u$ \emph{relative to} $B$, denoted by $\Supp_B(u)$, is defined as the set $\Supp_B(u)=\{i\in \Lambda\ \vert \  \alpha_i \neq 0\}$. 
An ideal $I$ of an evolution $\K$-algebra $A$ is called \textit{evolution ideal} if it admits a natural basis. In the case that this natural basis can be extended to the whole algebra we say that $I$ is a \textit{basic ideal}.  An ideal $I$ of a commutative $\K$-algebra $A$ satisfies  \textit{the absorption property} if $x A\subset I$ implies $x\in I$.

Following the philosophy of \cite{EL1}, we denote by $E=(E^0,E^1,r_E,s_E)$ the \emph{directed graph associated to an evolution algebra} $A$ \emph{relative to a natural basis} $B=\{e_i\}_{i\in \Lambda}$ where the set $E^0=\{e_i\}_{i \in \Lambda}$ and $f \in E^1$ with $s(f)=e_i$ and $r(f)=e_j$ if the $j$-th coordinate of $e_i^2$ relative to $B$ is nonzero.
So,  we sometimes identify $B$ with $E^0$.

Observe that the directed graph associated to an evolution algebra is row-finite and satisfies the Condition (Sing).

\section{Maximal ideals}\label{sec3}

In this section we define, on the one hand, certain subset of vertices related to an ideal, which will be a hereditary subset, and on the other hand, a subspace constructed from a hereditary subset, which turns out to be a basic ideal. We prove some properties of these two objects that leads us to recognize maximal ideals. Moreover, we identify this type of ideals in the graph associated to an evolution algebra. For perfect evolution algebras, all the maximal ideals come from ideals associated to maximal hereditary   subsets.
\begin{definition}\label{defHI}\rm 
For an ideal $I$ of an evolution $\K$-algebra $A$ with natural basis $B=\{e_i\}_{i \in \Lambda}$ and $E$ its associated directed graph, denote by $H_I$ the subset of $E^0$ given as 
$H_I:=\{e_i\in E^0\colon e_i^2\in I\}$. For any hereditary subset $H\subset E^0$ define the subspace $I_H:=\oplus_{e_i\in H} \K e_i =\span{H}$. Notice that $I_{\emptyset}=\{0\}$.
\end{definition}

\begin{lemma}\label{bocata}
Let $A$ be an evolution $\K$-algebra with natural basis $B=\{e_i\}_{i\in \Lambda}$ and $E$ its associated directed graph. We have the following assertions:

\begin{enumerate}
    \item If $H$ is a hereditary subset of $E^0$, then $I_H$ is a basic ideal of $A$. 
\item  If $I$ is an ideal of $A$, then $H_I$ is a hereditary subset of $E^0$.
\end{enumerate}

\end{lemma}

\begin{proof}
The first item is straightforward. For the second item, we take $e_j\in H_I$ being a non-sink and assume first that $f\in E^1$ has $s(f)=e_j$ and $r(f)=e_k\in E^0$. Then $e_j^2=ke_k+\sum_i k_i e_i$ where $k, k_i \in \K$  with $k \neq 0$ and $e_i\in B\setminus\{e_k\}$. Since $e_j^2\in I$ we have $e_ke_j^2\in I$. So $I\ni e_ke_j^2=ke_k^2$ which gives $e_k\in H_I$ because $k\ne 0$. The general situation is that there is 
a path $\lambda=f_1\cdots f_q$ with $s(f_1)=e_j$ and $r(f_k)=e_k$.
Applying the previous particular case we get that  $r(f_1)\in H_I$ and iterating this argument we finally get that $e_k\in H_I$.

\end{proof}
Thus for any ideal $I$ of an evolution $\K$-algebra, the particular set $H_I$ constructed above induces an ideal $I_{H_I}\supset I$ that will be proved in Proposition \ref{sabado1}.

\begin{remark} \rm
    An interesting property is that $
I_{T(S)}=\bigoplus_{e_i\in T(S)}\K e_i\triangleleft A$
for any $S\subset E^0$. 

\end{remark}

We denote by $\Id$ the set of all ideals of an evolution $\K$-algebra $A$
and we still denote by $\H$ the set of all hereditary subsets of the directed graph associated to  $A$ fixing a natural basis.

In the below Proposition we give some properties that can be observed. 

\begin{proposition}\label{sabado1}
Let $A$ be an evolution  $\K$-algebra with $B=\{e_i\}_{i\in \Lambda}$ a natural basis. Then:
\begin{enumerate}
    \item If $H, H'\in\H$, then $H\cap H'$, $H\cup H'\in\H$.
    \item If $H, H'\in\H$, then $I_{H\cap H'}=I_H\cap I_{H'}$. 
  If $I,I'\in \mathcal{I}$, then $H_{I\cap I'}=H_I\cap H_{I'}$.
    \item $I_{H\cup H'}=I_{H}+I_{H'}$. If $H\cap H'=\emptyset$, then $I_{H\cup H'}=I_{H}\oplus I_{H'}$.
     
    \item $I\subset I_{H_I}$ and $H\subset H_{I_H}$. 
    \item $I_H=A$ if and only if  $H=E^0$. 
    \item $I_{H_I}=A$ if and only if  $A^2\subset I$. 
     \item Assume $H\in \H$, $H_{I_H}=H$ if and only if $H$ is a saturated set.
     \item If $H_I=I\cap B$, then $H_I$ is a saturated set.
     \item If $H \in \H$, then $H=I_H \cap B$.
\end{enumerate}

\end{proposition}

\begin{proof}
The items (1), (2), (3), (5), (6) and (9) can be easily checked. 
For item (4), consider $0\ne x=\sum \lambda_i e_i \in I$ where  $\lambda_i \in \K^{\times}$. So, $e_ix=\lambda_i e_i^2 \in I$ implies $e_i^2\in I$. By definition, $e_i \in H_{I}$ and hence $x\in I_{H_I}$. 
Now, consider $e_j\in H$. So $T(e_j)\subset H$. Since $e_j^2=\sum \lambda_i e_i$ with $e_i\in H$ we have that $e_j^2\in I_H$ which implies that $e_j\in H_{I_H}$. For item (7), we first assume that $H=H_{I_H}$ and consider $r(s^{-1}(e_j))\subset H$ with $e_j\in E^0$ and $e_j$ is not a sink. Then $e_j^2=\sum_{e_i\in H} k_i e_i \in H$,  $k_i\in \K$ which implies that $e_j^2\in I_H$ and $e_j\in H_{I_H}= H$.  
Now, we assume that $H$ is saturated. We have to prove $H_{I_H}\subset H$. Consider $e_j\in H_{I_H}$. This implies that $e_j^2\in I_H$ and $e_j^2=\sum_{e_i\in H} k_i e_i$, $k_i\in\K$. This means that $r(s^{-1}(e_j))\subset H$ so, $e_j\in H$ as we wanted. For item (8), let $e\in B$ such that $0\neq e^2$ ($e$ is a regular vertex) and $r(s^{-1}(e)) \subset H_I$. We can write $e^2=\sum_{e_i\in H_I}\l_ie_i$. As $H_I=I \cap B$, then $e^2\in I$. Therefore, $e \in H_I$.
\end{proof}

Notice that the inclusions of item (4) Proposition \ref{sabado1} may be strict as the following example shows:
\begin{example}\rm
Consider $A$ an evolution $\K$-algebra with natural basis $B=\{e_1,e_2\}$ such that $e_1^2=e_1+e_2, e_2^2=-e_1-e_2$ and the ideal $I=\span{\{e_1+e_2\}}$. We have that $H_I=\{e_1,e_2\}=B$, so $I_{H_I}=A$. Hence, $I\subsetneq I_{H_I}$.
Now, consider $A'$ an evolution $\K$-algebra with natural basis $B'=\{e_1',e_2'\}$ such that $e_1'^2=e_1', e_2'^2=e_1'$ and the hereditary set $H=\{e'_1\}$. Then $I_H=\span{\{ e'_1\}}$  and $H_{I_H}=\{e'_1,e'_2\}\supsetneq H$.

\end{example}

The converse of item (8) Proposition \ref{sabado1} is not true in general. Let us see the next example.

\begin{example} \rm
    Let $A$ be an evolution $\K$-algebra and $B=\{e_1,e_2,e_3\}$ a natural basis with  $e_1^2=e_2^2=e_1+e_2$ and $e_3^2=e_3$. If we consider the ideal $I=\span{\{e_1+e_2\}}$, then $H_I=\{e_1,e_2\}$ which is a hereditary and saturated set. But $I\cap B=\emptyset\neq H_I$.
\end{example}

\begin{remark} \rm
Consider $H$ a maximal hereditary subset of a natural basis $B$ of an evolution $\K$-algebra $A$. One can observe that either $\{e_i\in B\,\colon\,e_i^2=0\}\subset H$ or at most there is one sink $e_i$ such that $H\cup \{e_i\}=B$. 
\end{remark}

One can think that every maximal ideal comes from a maximal hereditary set, but it is not true in general as we show in the following example.

\begin{example}\rm 

Consider $A$ an evolution $\K$-algebra with natural basis $B=\{e_1,e_2\}$ such that $e_1^2=e_1+e_2,e_2^2=e_1+e_2$. $B$ does not have maximal hereditary  subsets except the empty set but $I=\span{\{e_1+e_2\}}$ is a maximal ideal.
\end{example}

We provide a first approach to maximal ideals of an evolution $\K$-algebra in the following proposition.

\begin{proposition}\label{sabado2}
Let $A$ be an evolution  $\K$-algebra. Then:
\begin{enumerate}
        \item Let $I$ be an ideal of $A$ with $A^2\subset I$. Then $I$ is a maximal ideal of $A$ if and only if $I$ is a subspace of codimension $1$.
       \item If $I$ is a maximal ideal of $A$ with $A^2\not\subset I$, then $I=I_{H_I}$.
       \item If $I_H$ is a maximal ideal, then $H$    is a maximal element of $\H$. A kind of converse is:
    if $H$ is maximal in $\H$      and $I_H$ is not contained in a  subspace $S$ with $A^2\subset S\subsetneq A$, then $I_H$ is maximal.
    \end{enumerate}

\end{proposition}

\begin{proof}
 For item (1), since $ I$ is a maximal ideal, there exists $x \in A$ such that $x\notin I$. Note that $I + \K x$ is an ideal of $A$ (we recall that if $S$ is a vector subspace of an algebra $A$ with $A^2 \subset S \subset A$, then $S$ is an ideal of $A$) and by maximality $I + \K x = A$. Hence $A/I \cong \K x $. For the converse, suppose $J$ an ideal of $A$ with $I\subsetneq J$. Then $A^2\subset I\subsetneq J$ and $A/J\subsetneq A/I$. Since $\dim(A/I)=1$ we have that $A=J$.
For item (2),  we have that $I \subset I_{H_I}$ by item (4) of Proposition \ref{sabado1}. Since $I_{H_I} \neq A$ by item (6) of Proposition \ref{sabado1}, then by maximality $I=I_{H_I}$.
For the first part of item (3) suppose $H\subset H'$, hence $I_H\subset I_{H'}$. Since $I_H$ is maximal we have two cases: $I_{H'}=I_H$ or $I_{H'}=A$. In the first case, 
for any $e\in H$ we have $e\in I_H$ whence $e\in I_{H'}$ and this implies $e\in H'$. Thus $H\subset H'$ and similarly $H'\subset H$. So, $H=H'$. In the second case, by item (5) of Proposition \ref{sabado1}, $H'=E^0$.
For proving the other assertion, we take $H$ to be maximal and such
that $I_H$ is not contained in any subspace $S$ with $A^2\subset S\subsetneq A$. Assume that $I_H\subset J\ne A$, then $H\subset H_{I_H}\subset H_J$ so: either $H_J=H$, implying $J\subset I_{H_J}= I_H$ or $H_J=E^0$ implying $A^2\subset J$ by item (5) and (6) of Proposition \ref{sabado1}. In the last case we get the contradiction $A^2\subset J\subsetneq A$ and $I_H\subset J$.

\end{proof}

The converse of the item (2) and the converse of the first part of  item (3) of Proposition \ref{sabado2} is not true in general, that is, not all the maximal elements $H\in\H$ produce maximal ideals $I_H$ as shows the following example. 

\begin{example}\rm
Let $A$ be an evolution $\K$-algebra with natural basis $B=\{e_1,e_2,e_3,$\\$e_4\}$ such that $e_1^2=e_1+e_2,e_2^2=e_2,e_3^2=e_1+e_3+e_4,e_4^2=e_1+e_3+e_4$. Then, $H=\{e_1,e_2\}$ is a
maximal hereditary subset of $B$. However, $I_H=\K e_1\oplus \K e_2$ is not a maximal ideal since $J=\span{\{e_1,e_2,e_3+e_4\}}$ is an ideal with $I_H\subset J$ and $J\neq A$. Observe that if we take the ideal $I=I_H$ then $I=I_{H_I}$, moreover $A^2\not\subset I_H$ and $I$ is not a maximal ideal.
\end{example}

\begin{corollary}\label{blanco}
Let $A$ be a perfect evolution $\K$-algebra. Let $I$ be a maximal ideal of $A$, then there exists a maximal hereditary subset $H$ such that $I=I_{H}$. Furthermore, $H=H_I$ and $I$ is a basic ideal.
\end{corollary}
\begin{proof}
It is straightforward of item (2) and (3) Proposition \ref{sabado2}.
\end{proof}

\begin{remark}\label{remark13}\rm
From Proposition \ref{sabado2}, one can see that the maximal ideals can be obtained from the graph associated to an evolution algebra except for those maximal ideals, $I$, with $A^2\subset I \subset A$ which must be investigated directly. Summarizing, we can say that the maximal ideals of an evolution algebra exists in two flavours: 
\begin{enumerate}
    \item  Those subspaces of codimension $1$ of $A$ containing $A^2$.
    \item Those ideals $I_H$ with $H\in\H$ maximal, and such that there is no subspace $S$ with $A^2\subset S\subsetneq A$ such that $I_H\subset S$. 
\end{enumerate}
 To illustrate how this works in a particular evolution algebra consider the following examples. Concretely, in Example \ref{ej0} we can see that the evolution algebra has ideals of both types and Example \ref{ej1} gives infinitely many maximal ideals containing $A^2$.
\end{remark}

\begin{example}\label{ej0}\rm
Consider $A$ an evolution $\K$-algebra with $B=\{e_i\,\colon\, i\in \N^* \}$ a natural basis and $e_i^2=e_i+e_{i+1}$ for all $i\in \N^*$ with the following associated graph.
\bigskip

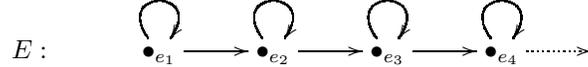
\begin{equation*} 
   E:   \xymatrix{
     & {\bullet}_{e_{1}}\ar@(ul,ur) \ar[r]  & {\bullet}_{e_{2}} \ar[r] \ar@(ul,ur) & {\bullet}_{e_{3}} \ar[r] \ar@(ul,ur) & {\bullet}_{e_{4}}\ar@(ul,ur)  \ar@{.>}[r] &    \\
           }
           \end{equation*}
The only maximal hereditary subset is $H=\{e_i\,\colon \, i\geq 2\}$. So, we have the maximal ideal $I_H=\oplus_{i\geq 2}\K e_i$ since $A/I_H\cong \K \Bar{e_1}$. Moreover, by Proposition \ref{sabado2}, there does not exist any other maximal  ideal except $A^2=\span{\{ e_i+e_{i+1}\,\colon\, i\geq 1\}}$. Indeed, if there exists another maximal ideal $J$ with $J\neq A^2$ necessarily $A^2 \not\subset J$. By item (2) of Proposition \ref{sabado2} we get that $J=I_{H_J}$. Clearly $H_J=H$ and then $J=I_H$. Observe that $A^2$ can not write as $I_{H'}$ with $H'$ a maximal hereditary subset.

\end{example}

\begin{example}\label{ej1}\rm
Consider $A$ an evolution $\K$-algebra with natural basis $B=\{e_1,e_2,$\\
$e_3,e_4,e_5,e_6\}$ such that $e_1^2=e_2,e_2^2=e_2,e_3^2=e_4+e_5,e_4^2=0,e_5^2=e_6, e_6^2=0$. The ideals $I_1=\span{\{ e_2,e_3,e_4,e_5,e_6\}}$ and $I_2=\span{\{ e_1,e_2,e_4,e_5,e_6\}}$ are  maximal ideals that come from the maximal hereditary subsets $H_1=\{e_2,e_3,e_4,e_5,e_6\}$ and $H_2=\{e_1,e_2,e_4,e_5,e_6\}$, respectively,
with $A^2\subset I_1,I_2$ ($A^2=\span{\{ e_2,e_4+e_5,e_6\}}$). Moreover, there are infinite maximal ideals $I'$ such that $A^2\subset I'\subset A$, indeed, $I'=\span{\{ e_2,e_4+e_5,e_6,\alpha e_1+\beta e_3,\gamma e_1+\delta e_3\}}$ with $\det \tiny\begin{pmatrix}
\alpha & \beta\\
\gamma & \delta
\end{pmatrix}\neq 0$.
\end{example}

\begin{proposition}
Consider $A$ an evolution $\K$-algebra with natural basis $B=\{e_i\}_{i \in \Lambda}$.
Let $I\triangleleft A$ be a maximal ideal. Then:

\begin{enumerate}
    \item If $e\in B\setminus I$, then for any $e_i\in B$ either $e_i\in T(e)$ or $e_i^2\in I$. In other words, $B=T(e)\cup H_I$. 
    \item If $\dim(A/I)\ne 1$, then for any $e\in B$
we have $$e^2\in I\quad\Rightarrow\quad e\in I,$$
that is to say, $H_I=I\cap B$.
\end{enumerate}

\end{proposition}

\begin{proof}
For item (1), assume that $e_i\notin T(e)$.
Since $e\notin I$ and $I$ is a maximal ideal, then $I+I_{T(e)}=A$. We can write $e_i=x+z$ with $x\in I$ and $z=\sum_j \lambda_j e_j$ where $\lambda_j\in \K$ and $e_j\in T(e)$.
Therefore $e_i^2=e_i x\in I$, so $e_i \in H_I$. For item (2), we have $I\subset I+\K e\triangleleft A$, then either $I+\K e=I$ or $I+\K e=A$. In the first case, $e\in I$ as we claimed. In the second one, there are two possibilities:
\begin{enumerate}
    \item If $I\cap \K e\ne 0$, then again $e\in I$.
    \item If $I\cap \K e=0$, then $A/I\cong \K e$ as $\K$-vector spaces. Hence $\dim(A/I)=1$,  a contradiction.
\end{enumerate}
 So, $H_I \subset I \cap B$, since we always have $I\cap B\subset H_I$, then $H_I=I\cap B$. 
\end{proof}

In the Section \ref{sec4}, we will see in Proposition \ref{tarde} that $I$ has the absorption property in the previous settings.

The following example shows an infinite-dimensional evolution $\K$-algebra without either maximal ideals or maximal hereditary subsets in the set of vertices of its associated directed graph.
\begin{example}\rm
Consider $A$ an evolution $\K$-algebra with natural basis $B=\{e_i\colon i\in\mathbb{Z}\}$ and $e_i^2=e_{i+1}$ for all $i\in \mathbb{Z}$. The associated graph is
\bigskip

\begin{equation*} 
   E:   \xymatrix{
     & \ar@{.>}[r] & {\bullet}_{e_{-1}} \ar[r]  & {\bullet}_{e_{0}} \ar[r]  & {\bullet}_{e_{1}} \ar@{.>}[r] &    \\
           }
           \end{equation*}
This algebra is perfect. Any $H_n=\{e_i\colon i\ge n\}$ is a hereditary subset and there is no maximal hereditary subsets. Moreover,  $I_{H_n}=\oplus_{i\ge n}\K e_i$ is an ideal so there is no maximal ideals in this algebra.

\end{example}

In general, the existence of maximal ideals in arbitrary algebras is not guaranteed. However, observe that if a $\K$-algebra $A$ is finitely-generated, then a standard application of Zorn's Lemma implies that $A$ has maximal ideals. For example:

\begin{example}\label{felicidades} \rm
Consider $A$ an evolution $\K$-algebra with natural basis $B=\{e_i\,\colon\, i\in \N^* \}$ and $e_1^2=e_1+e_{2}$ and $e_i^2=e_{i+1}$ for all $i\geq 2$ with the following associated graph.

\bigskip

\begin{equation*} 
   E:   \xymatrix{
     & {\bullet}_{e_{1}}  \ar@(ul,ur) \ar[r]  & {\bullet}_{e_{2}}  \ar[r] & {\bullet}_{e_{3}}  \ar[r] & {\bullet}_{e_{4}}   \ar@{.>}[r] &    \\
           }
           \end{equation*}

$A$ is generated by $\{e_1,e_2\}$ and the ideal generated by $e_2$ is maximal.
\end{example}

\section{Finitely generated evolution algebras}\label{sec5}
We devote this section to study conditions implying the finitely-generated nature of an evolution algebra. 

Let  $A$ be an arbitrary dimensional evolution $\K$-algebra (so possible infinite-dimensional).  Assume that  $A$  is finitely-generated as $\K$-algebra.
Take $\{g_1,\ldots, g_q\}$ to be a finite system of (algebra) generators of  $A$. Fix a natural basis $B=\{e_i\}_{i\in\Lambda}$ 
 of $A$ and consider it as the vertex set of its associated directed graph.

For any $j=1,\ldots, q$, there is no loss of generality if we write $g_j=\sum_{k=1}^n \lambda_{jk} e_{k}$ with $\lambda_{jk}\in\K^{\times}$ and a suitable $n$. 
 Denote by $H$ the hereditary closure of $\{e_{1},\ldots,e_{n}\}\subset B$. Since any $e_{l}\in H$ if $l\in\{1,\ldots,n\}$,
 we have $e_{l}\in I_H$ whence $g_j\in I_H$ for any $j$. We conclude that any finite product $(\cdots(g_{j_1}g_{j_2})\cdots) g_{j_m}\in I_H$ for any $j_1,\ldots,j_m\in\{1,\ldots,q\}$. Thus the whole algebra $A$ agrees with $I_H$. Then, $H=B=E^0$ by item (5) of Proposition \ref{sabado1}. So, we have proved the following Lemma:

\begin{lemma}\label{azul}
 If $A$ is an arbitrary dimensional evolution $\K$-algebra but finitely-\hyphenation{ge-ne-ra-ted}generated as algebra, then the directed graph $E$ associated to $A$
 satisfies that $E^0$ agrees with the hereditary closure of a finite set. 

 \end{lemma}

In Example \ref{felicidades}, the $\K$-algebra $A$ is infinite-dimensional but a system of generators is $\{e_1,e_2\}$, besides we have that its hereditary closure is $\overline{\{e_1,e_2\}}=\{e_i\}_{i\ge 1}$ in agreement with Lemma \ref{azul}.

Observe that the condition proved in Lemma \ref{azul} is sufficient  but not necessary as shown in the next example.

\begin{example}\rm 
Consider an evolution $\K$-algebra $A$ with natural basis $B=\{e_i\}_{i\in \N^*}\\ \cup\{e_i'\}_{i\in\N^*}$ and $e_i^2=e_{i+1}+e_i'$ for all $i\in \N^*$. Notice that the set of vertices of the associated graph $E$ is $E^0=\overline{\{e_1\}}$ but $A$ is not finitely-generated. To see this, let us take the subset $\{e_1,\ldots, e_n,e_1',\ldots, e_{n-1}'\}$ of $B$. This subset generates the elements $S=\{e_1,\ldots, e_n,e_1',\ldots, e_{n-1}', e_{n+1}+e_n',e_{n+2}+e_{n+1}',\ldots\}$ but $e_{n+1}$ can not be expressed as linear combination of the elements in $S$.


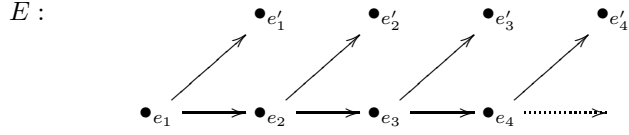
\begin{equation*} 
   E:   \xymatrix{
   &  & {\bullet}_{e_{1}'} & {\bullet}_{e_{2}'} & {\bullet}_{e_{3}'} & {\bullet}_{e_{4}'} \\
     & {\bullet}_{e_{1}}  \ar[ur] \ar[r]  & {\bullet}_{e_{2}} \ar[ur] \ar[r] & {\bullet}_{e_{3}} \ar[ur] \ar[r] & {\bullet}_{e_{4}} \ar[ur]  \ar@{.>}[r] &    \\
           }
           \end{equation*}
\noindent
A maximal ideal here is $I=\span{\{e_i': i\in \N^*, \,\, e_{j+1}:j\in \N^*\}}$. In fact, it is unique.
\end{example}

\begin{proposition}\label{noche}
Consider $A$ an arbitrary dimensional evolution $\K$-algebra with natural basis $B$ and $E$ its associated directed graph. Assume that  $E$ contains finitely many bifurcations. $A$ is finitely-generated as algebra if and only if $E^0=\overline{S}$ for a finite subset $S\subset E^0$.
\end{proposition}

\begin{proof}
If $A$ is finitely-generated, then there exists a finite set $S\subset E^0$ such that $E^0=\overline{S}$ by Lemma \ref{azul}.
Conversely, assume that $E$ contains only finitely many bifurcations and connected components and $E^0=\overline{S}$ for a finite set $S$.
Let $V\subset E^0$ be the set of all bifurcations. Then take $S'=S\cup V\cup\{r(v)\colon v\in V\}$. We still have $\overline{S'}=E^0$ and $S'$ is finite. If $u\notin S'$ and 
$f\in E^1$ with $s(f)\in S'$ and $r(f)=u$, then $s(f)$ is not a bifurcation so $s(f)^2=k u\ne 0$ with $k \in \K^{\times}$. Thus $u$ is in the subalgebra generated by $S'$. Now, assume that for any $u\notin S'$ with an arbitrary path $\lambda$ (without bifurcations) of length $n\geq 1$ such that $s(\lambda)\in S'$ and $r(\lambda)=u$, we have that $u$ is in the subalgebra generated by $S'$. Consider next an element $u\notin S'$ and a path $\mu$ (without bifurcations) of length $n+1$, whose source is in $S'$ and $r(\mu)=u$. Then 
writing $\mu=\lambda f$ with $\lambda$ of length $n$ we have:
\begin{enumerate}
    \item If $s(f)\in S'$ we have (because of the first induction step) that $u$ is in the subalgebra generated by $S'$. 
    \item If $s(f)\notin S'$, then $s(f)$ is in the subalgebra generated by $S'$ by the induction hypothesis. Furthermore, $s(f)$ is not a bifurcation whence $s(f)^2=h u\ne 0$ with $h \in \K^{\times}$. This implies that $u$ is in the subalgebra generated by $S'$.
\end{enumerate}
\end{proof}
The following examples concern  Proposition \ref{noche}.
\begin{example}\rm
Consider $A$ an evolution $\K$-algebra with natural basis $B=\{e_1,e_2,$$\\
$$e_3\}\cup \{e_i'\}_{i\in \N^*}$ with $e_1^2=e_2+e_3, e_2^2=0,e_3^2=0,e_i'^2=e_1$ for all $i\in \N^*$. Let $E$ be the associated directed graph to $A$. Observe that we only have just one bifurcation  (at $e_1$) but $A$ is not finitely-generated. Moreover, there does not exist a finite subset $S\subset E^0$ such that $E^0=\overline{S}$. 
\bigskip
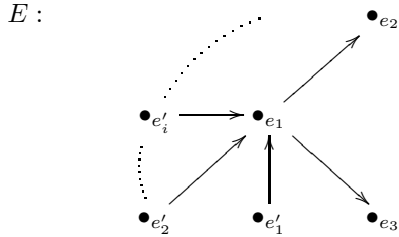
\begin{equation*} 
   E:   \xymatrix{
   & {}    &     & {\bullet}_{e_{2}}   \\
   & {\bullet}_{e_{i}'} \ar@{.}@/^/[ur] \ar[r]& {\bullet}_{e_{1}} \ar[ur]  \ar[dr] &     \\
     & {\bullet}_{e_{2}'}  \ar[ur]  \ar@{.}@/^/[u]&   {\bullet}_{e_{1}'} \ar[u]  &   {\bullet}_{e_{3}}    \\
           }
           \end{equation*}

\noindent
We can consider the family of maximal ideals given by $I_k=\span{\{e_1,e_2,e_3,e'_{i}\colon \\ i\in \N^*\}\setminus \{e'_k\}}$ for all $k\in \N^*$.

\end{example}

\begin{example}\rm 
Consider $A$ an evolution $\K$-algebra with natural basis $B=\{e_0\}\cup\{e_i\}_{i\in \N^*}$ such that $e_0^2=0,e_i^2=e_0+e_{i+1}$ for all $i\in \N^*$. Observe that its associated graph  $E$, has infinite bifurcations but $A$ is finitely-generated by the subset $\{e_0,e_1\}$. Moreover, the set of vertices of $E$ is $E^0=\overline{\{e_1\}}$.
\bigskip
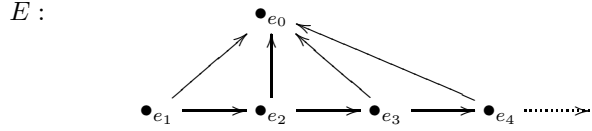
\begin{equation*} 
   E:   \xymatrix{
   &  & {\bullet}_{e_{0}} &  &  &  \\
     & {\bullet}_{e_{1}}  \ar[ur] \ar[r]  & {\bullet}_{e_{2}} \ar[u] \ar[r] & {\bullet}_{e_{3}} \ar[ul] \ar[r] & {\bullet}_{e_{4}} \ar[ull]  \ar@{.>}[r] &    \\
           }
           \end{equation*}
\noindent
Observe that $I=\span{\{e_0, e_{i+1}\colon i\in \N^*\}}$ is the unique maximal ideal of $A$.

\end{example}

\section{Ideals having the absorption property and Galois connection}\label{sec4}
In this section, we first focus on studying when an algebra has ideals having the absorption property. In addition, we characterizes the absorption property in terms of the graph. Then, we prove that the set of all hereditary and saturated subsets and the set of all ideals having the absorption property with certain maps form a Galois connection. For perfect evolution algebras, we obtain that this connection is monotone.

\subsection{Ideals having the absorption property}

Now, we give some statements related to ideals that have the absorption property.
\begin{proposition}\label{tarde} Let $A$ be a commutative $\K$-algebra. If $I$ is a maximal ideal of $A$ such that either $\dim(A/I)\ne 1$ or $A^2 \not\subset I$, then $I$ has the absorption property.  Consequently, if $A$ is a perfect commutative $\K$-algebra and $I$ is a maximal ideal, then  $I$ has the absorption property.
\end{proposition}
\begin{proof}
Assume that $xA\subset I$. Since $I \subset I + \K x$, $I + \K x$ is an ideal of $A$ and $I$ is a maximal ideal, we have two possibilities: $I + \K x = A$ or $I + \K x = I$. Now, we suppose that $\dim(A/I)\ne 1$. Then,  necessarily $I = I + \K x$ because otherwise $\dim(A/I)=1$. So, $x \in I$. Now, we assume that $A^2 \not \subset I$. Therefore, $I + \K x = A$ and this implies $A^2\subset I+\K x^2\subset I$, a contradiction. Then $I + \K x = I$, so $x \in I$. 
\end{proof}

\begin{proposition}\label{sabado3}
Let $A$ be a non-degenerate evolution $\K$-algebra and $H$ a hereditary and saturated set, then $I_H$ has the absorption property.
\end{proposition}
\begin{proof}
    Let $B=\{e_i\}_{i\in \Lambda}$ be a natural basis of $A$. Consider $xA\subset I_H$ with $0\neq x=\sum_{i\in\Supp(x)}\lambda_i e_i$. Since $0\neq x e_i=\lambda_i e_i^2\in I_H =\oplus_{e\in H}\K e$ for all $i\in \Supp(x)$ we have that $r(s^{-1}(e_i))\subset H$  with $e_i^2\neq 0$. So, $e_i\in H$ for all $i\in \Supp(x)$ because $H$ is a saturated set, implying that $x\in I_H$.
\end{proof}

The condition of non-degenerate can not be eliminated as we show in the next example. 
\begin{example}\rm
Consider $A$ an evolution $\K$-algebra with natural basis $B=\{e_i\}_{i=1}^4$ such that $e_1^2=e_2^2=e_3^2=e_3$ and $e_4^2=0$. Observe that $A$ is degenerate. The subset $H=\{e_1,e_2,e_3\}$ is hereditary and saturated. If we take $x=e_1+e_4$, then $\K e_3=xA\subset I_H$ but $x\notin I_H$.
\end{example}

\begin{proposition} \label{propnueva}
    Let $A$ be an evolution $\K$-algebra and $H\in \H$. If $I_H$ has the absorption property, then $H$ is a saturated set. 
\end{proposition}

\begin{proof}
    Let $B=\{e_i\}_{i\in \Lambda}$ be a natural basis of $A$. Suppose that $e\in B$ with $e^2\neq 0$ and $r(s^{-1}(e))\subset H$. Then we have that $e^2=\sum_{e_j\in H}\l_je_j \in I_H$. Therefore $e A \subset I_H$ and so $e \in I_H$ because $I_H$ has the absorption property. By item (9) of Proposition \ref{sabado1}, we get $e \in H$.
\end{proof}

Observe that if $A$ is a non-degenerate evolution $\K$-algebra, then $H$ is a hereditary and saturated subset if and only if $I_H$ has the absorption property.

The following theorem characterizes the ideals having the absorption property.

\begin{theorem}\label{carrillada}
    Let $A$ be an evolution $\K$-algebra with natural basis $B=\{e_i\}_{i\in\Lambda}$. The following assertions are equivalent:
    \begin{enumerate}
    \item $I$ has the absorption property.
    \item $H_I=I\cap B$. 
   \item  $I=I_{H_I}$.
    \end{enumerate}
\end{theorem}

\begin{proof}
Suppose that (1) holds true and take $e\in H_I$. So $e^2\in I$ and $e A\subset I$. Since $I$ has the absorption property, $e\in I$ and $e\in I\cap B$. Now, if $e\in I\cap B$, then $e^2\in I$. Hence, $e\in H_I$. Now, assume that (2) is true. We know that $I\subset I_{H_I}$ by item (4) of Proposition \ref{sabado1}. Let $x=\sum_{e_i\in H_I} \lambda_i e_i \in I_{H_I}$. Since $H_I=I\cap B$, then $x\in I$. Next, suppose that $I=I_{H_I}$. We consider $0\neq x\in A$ such that $x A \subset I$. We write $x=\sum_{i\in \Supp(x)}\l_ie_i$. So, $xe_i=\l_ie_i^2\in I$ for all $i\in \Supp(x)$. Therefore $e_i\in H_I \subset I_{H_I}=I$ for all $i\in \Supp(x)$ and this implies that $x\in I$.

\end{proof}

Proposition \ref{sabado3} and Theorem \ref{carrillada} give a way to find ideals having the absorption property just looking for the hereditary and saturated sets of the associated graph in the case of  non-degenerate evolution algebras.

\begin{corollary}
    Let $A$ be an evolution $\K$-algebra with natural basis $B$. If $I$ is a maximal ideal with $A^2\not \subset I$, then $I$ has the absorption property and $H_I=I\cap B$ is a saturated set.
\end{corollary}

We see that if $A^2\subset I$, the previous corollary is not true  in the next example .
\begin{example}\rm
Assume $A$ is an evolution $\K$-algebra with natural basis $B=\{e_i\}_{i=1}^4$ such that $e_1^2=e_2^2=e_3^2=e_4^2=e_3$. Consider the maximal ideal $I=\oplus_{i=1}^3\K e_i$ of codimension $1$.   If we take $x=e_4$, then $\K e_3=xA\subset I$ but $x\notin I$, so $I$ does not have the absorption property. Moreover, since $A^2=I$ and $I=I_H$ with $H=\{e_1,e_2,e_3\}$, we have that $H_I=B\neq I\cap B=\{e_1,e_2,e_3\}$.
\end{example}

\begin{theorem}\label{fiesta}
    Let $A$ be a  finite-dimensional perfect evolution $\K$-algebra and $I$ an ideal of $A$. Then $I=I_{H_I}$ and $I$ has the absorption property.  
\end{theorem}
\begin{proof}
    First, we prove that $\dim(I)=\vert H_I\vert$. For this purpose, we consider $B=\{e_i\}_{i=1}^n$ a natural basis of $A$, $H_I=\{e_i\,\colon\, e_i^2\in I\}$ and $\{v_j\}_{j=1}^s$ a basis of $I$. Let us denote $N:=\vert H_I\vert$ and $H_I=\{e_i\}_{i=1}^N$.  Since  $\{e_i^2\}_{i=1}^N \subset I$ is a linear independent set therefore $N \leq \dim(I)=s$. Suppose that $N < s$. This implies that there exists $l\in\{1,\ldots,s\}$ with $v_l=\sum_{k=1}^N \mu_{kl}e_{k} +\sum_{k=N+1}^n \mu_{kl}e_{k}\in I $ such that $\mu_{jl}\neq 0$ for some  $j\in \{N+1,\ldots,n\}$, otherwise, we would have that the linearly independent set $\{v_j\}_{j=1}^s$ can be written as a linear combination of a smaller linearly independent set $\{e_k\}_{k=1}^N$ which is impossible. 
    But, $v_le_j=\mu_{jl} e_j^2 \in I$, a contradiction because $e_j \notin H_I$. Then $\dim(I)=\vert H_I\vert$. Since $I \subset I_{H_I} $ necessarily $I=I_{H_I}.$ For the second statement we apply Theorem \ref{carrillada}.
\end{proof}

Observe that all  nonzero ideals of a finite-dimensional  perfect evolution $\K$-algebra  have the absorption property. Furthermore, they verify $I=I_{H_I}$.

As a consequence of Theorem \ref{fiesta} we recover \cite[Proposition 4.2]{boudi20}, with a different approach, concerning  the ideals of a finite-dimensional perfect evolution $\K$-algebra.

\begin{corollary}\label{propbasic}
    Let $A$ be a perfect evolution $\K$-algebra of finite dimension. Every nonzero ideal $I$ of $A$ is a basic ideal.
\end{corollary} 
\begin{proof}
   Let $B=\{e_i\}_{i=1}^n$ be a natural basis of $A$. By Theorem \ref{fiesta}, $I=\span{H_I}$ and $H_I \subset B$ and this implies that  $I$ is a basic ideal by item (1) of Lemma \ref{bocata}. 
\end{proof}

\subsection{Galois Connection}

Recall that for two partially ordered sets  $(A, \leq)$ and $(B, \leq)$ we can define a \emph{Galois connection} between these sets consisting of two  functions: $F\colon A \to B$ and $G\colon B  \to A$, such that for all $a \in A$ and $b \in B$, we have $F(a) \leq b$ if and only if $a\leq  G(b)$. It is said that $F$ is the \textit{left adjoint} and $G$ is the \textit{right adjoint} of the Galois connection. If these functions are monotone, we have a \emph{monotone Galois connection}.

\begin{definition} \label{definicionapl} \rm
Let $A$ be an evolution $\K$-algebra. We define the maps $\mathfrak{f}:\H\to\Id$ and $\mathfrak{h}:\Id\to\H$ such that $H\mapsto I_H$ and $I\mapsto H_I$ respectively.
\end{definition}

Note that $\mathfrak{h}$ and $\mathfrak{f}$ are order preserving maps. Some results about these applications are the following.

\begin{proposition}\label{inyective}
The map $\mathfrak{f}$ is inyective and strictly monotone in the sense that $H\subsetneq H'$ implies $I_H\subsetneq I_{H'}$. 
\end{proposition}
\begin{proof}
We proved the inyectivity in the proof of item (3) of Proposition \ref{sabado2}. As a consequence, 

the map $H\mapsto I_H$ is strictly monotone.
\end{proof}

\begin{corollary}
Let $A$ be  a  finite-dimensional perfect evolution $\K$-algebra  with two ideals $I,J$ such that $H_I=H_J$, then $I=J$. In other words, $\mathfrak{h}$ is inyective.  
\end{corollary}
\begin{proof}
It is straightforward from Theorem \ref{fiesta}.
\end{proof}

Notice that the map $\mathfrak{h}$ is not necessarily inyective if the algebra is non perfect. To see this, consider the evolution $\K$-algebra with natural basis $B=\{e_1,e_2,e_3\}$ with $e_1^2=e_1+e_2,e_2^2=e_1+e_2,e_3^2=e_1+e_2+e_3$ and the ideals $I=\span{\{e_1+e_2\}}$ and $J=\span{\{e_1,e_2\}}$. We have $H_I=\{e_1,e_2\}=H_J$ and $I\neq J$.

Note that $(\H,\subset)$ and $(\Id,\subset)$ are two partially ordered sets. 

\begin{proposition}\label{sabado4}
    Let $A$ be an evolution $\K$-algebra and $B=\{e_i\}_{i\in \Lambda}$ a natural basis. If all the elements of $\Id$ have the absorption property, then $\mathfrak{f},\mathfrak{h}$ is a monotone Galois connection between $(\H,\subset)$ and $(\Id,\subset)$. 
\end{proposition}

\begin{proof}
    We have to prove that $I_H\subset I \Leftrightarrow H\subset H_I$ for any hereditary $H$ and any ideal $I$. Assume $I_H\subset I$ and $e\in H$. Then, $e\in I_H\subset I$ and $e^2\in I$. So, $e\in H_I$ as we wanted. Now, consider $H\subset H_I$ and $x=\sum_{e_i\in H}\lambda_i e_i \in I_H$. As $H\subset H_I$ hence $x\in I_{H_I}$ but since $I$ has the absorption property then $I_{H_I}=I$ by Theorem \ref{carrillada}. Therefore $x\in I$ as we wanted to show.
\end{proof}

This brings us to define $\H^*\subset \H$ to be the ordered subset of all hereditary and saturated sets and  $\Id^*\subset\Id$ the ordered subset of all ideals having the absorption property. 
\begin{theorem}\label{Galois}
    Let $A$ be a non-degenerate evolution $\K$-algebra and $B$ a natural basis. We have that:
    \begin{enumerate}     
    \item The restrictions $\mathfrak{f}\colon \H^*\to\Id^*$  and $\mathfrak{h}\colon \Id^*\to\H^*$ are a monotone Galois connection.
\item If $H_i \in \H^*$ for all $i\in \Gamma$ and $\cup_{i\in \Gamma}H_i\in \H^*$, then $I_{\cup_{i\in \Gamma}H_i}=\sum_{i\in \Gamma}I_{H_i}$.
    \item If $\{I_i\}_{i\in \Omega}$ is a collection of ideals having the absorption property, then $H_{\cap_{i\in \Omega} I_i}=\cap_{i\in \Omega}H_{I_i}$.
    \item  If $A$ is finite-dimensional with  $A=A^2$, then $\mathfrak{f},\mathfrak{h}$ is a monotone Galois connection between $(\H,\subset)$ and $(\Id,\subset)$.
    \end{enumerate}
\end{theorem}
\begin{proof}
    First, notice that $\mathfrak{f}$ and $\mathfrak{h}$ are well defined by Proposition \ref{sabado3}, Theorem \ref{carrillada} and item (8) of Proposition \ref{sabado1}. Now, applying Proposition \ref{sabado4} we get that they are a monotone Galois connection. We deduce item (2) and (3) from \cite[7.31 Proposition]{Galois}. Item (4) is straightforward from Theorem \ref{fiesta} and Proposition \ref{sabado4}.
\end{proof}

\section{Evolution algebras with simple associated graph}\label{sec6}

In this section, we will investigate the evolution $\K$-algebras with just one maximal ideal concerning its associated graph.
We start by studying perfect evolution algebras. In this case, by Corollary \ref{blanco}, we know that all maximal ideals come from maximal  hereditary  subsets. Proposition \ref{inyective} assures us that a maximal ideal comes from a unique  maximal  hereditary subset. So, we investigate the graphs associated to perfect evolution $\K$-algebras with just one maximal hereditary  subset.
Given a directed graph $E$ and $H\in\H$ we can define the quotient graph $E/H=(F^0,F^1,r_F,s_F)$ such that 
$F^0=E^0\setminus H$, $F^1:=\{f\in E^1\colon s(f),r(f)\notin H\}$, and 
$s_F=s_E\vert_{F^1}$, $r_F=r_E\vert_{F^1}$.

\begin{lemma}\label{miercoles}
    Consider  a graph $E$ and $H\in\H_E$. Then
    $$ \{H'\setminus H\, \ |\ \ H\subset H', H'\in \H_E\} \subset \H_{E/H}.$$
\end{lemma}
\begin{proof}
We assume $H'\in \H_E$, $H\subset H'$. Consider $u\in H'\setminus H$ and $\lambda\in \text{Path}(E/H)$ such that $s(\lambda)=u$ and $r(\lambda)=w\in E/H$. Since $u\in H'$ and $\lambda\in \text{Path}(E)$, then $w\in H'$. Since $w\in E/H$, $w\notin H$. This implies that $w\in H'\setminus H$ and $H'\setminus H\in \H_{E/H}$.

\end{proof}
\begin{definition}\label{defgrafosimple}\rm
We will say that a  directed graph $E$ is {\emph{simple}} if it has no hereditary subsets other than $E^0$ and $\emptyset$.
\end{definition}

\begin{proposition}\label{propsimple}
Let $A$ be a finitely-generated (as algebra) perfect evolution $\K$-algebra. Then $A$ is simple if and only if its associated directed graph $E$ is simple.
\end{proposition}

\begin{proof}
If $A$ is simple, then $\{0\}$ is the unique maximal ideal of $A$. Applying Corollary \ref{blanco}, there exists $H=\emptyset$ maximal hereditary subset such that $\{0\}=I_{\emptyset}$. This implies that $\H= \{ \emptyset, E^{0}\}$, so $E$ is simple. Now, consider a proper ideal $I$ of $A$. Since $A $ is finitely-generated, there exists a maximal ideal $J$ such that $I \subset J $. Again, by Corollary \ref{blanco}, $J$ can be written as $J=J_{H}$ with $H$ a maximal hereditary set. Due to the fact that $\H=\{\emptyset, E^0\}$, then necessarily $H=\emptyset$. Therefore $J=I=\{0\}$. So, $A$ is simple.
 \end{proof}

\begin{lemma}\label{tarari}
   Let $A$ be an evolution $\K$-algebra and $B$ a natural basis. Let $H$ be a   maximal hereditary set. Then $E/H$ is either an unique vertex or strongly connected. 
\end{lemma}

\begin{proof}
If $E/H$ has more than one vertex, then we consider $e_i, e_j \in E/H$. Therefore, by maximality of $H$,  $E/H$ has only one connected component. So, there exists a path $\mu$ such that either $s(\mu)=e_i$  and $r(\mu)=e_j$ or $s(\mu)=e_j$  and $r(\mu)=e_i$.  We can suppose, without loss of generality, that $s(\mu)=e_i$  and $r(\mu)=e_j$. Since $H\cup \{T(e_j)\}$ is a hereditary with $H \subset H\cup \{T(e_j)\} $, then by maximality of $H$ necessarily $H\cup \{T(e_j)\}=E^0$ so $e_i \in T(e_j)$. Therefore $E/H$ is strongly connected.
    
\end{proof}

\begin{theorem}\label{teo64}
Let $E$ be a directed  graph and $H\in\H_E$. 
\begin{enumerate}
\item $H$ is maximal if and only if $E/H$ is simple.
\item If the graph $E$ is simple, then it is either an isolated vertex or  a graph with no sources or sinks. Furthermore, it is strongly connected.  If $E^0$ is finite, then $E$ is simple if and only if there is a closed path visiting all the vertices of the graph.
\item  Let $A$ be an evolution algebra with associated graph $E$. If $H\in\H_E$, then the evolution algebra $A/I_H$ has associated graph $E/H$ relative to a certain natural basis.
\end{enumerate}
\end{theorem}
\begin{proof}

For item (1), suppose $H$ is maximal, then $E/H$ is either an unique vertex or strongly connected  by Lemma \ref{tarari}. So, $E/H$ is simple.
Now, suppose that $E/H$ is simple. Then $\H_{E/H}=\{\emptyset, E^0 \setminus H\}$. If there exists $H'\in \H_E$ such that $H \subsetneq H'$ by Lemma \ref{miercoles} $H' \setminus H \in \H_{E/H}$. So, necessarily $H'\setminus H=E^0\setminus H$ and we get that $H'=E^0$, hence $H$ is maximal. For item (2), if we have an isolated vertex we are done. Assume we have more than one vertex and take $u\in E^0$ a source. Then there exists $f\in s^{-1}(u)$ with $T(r(f))\in \H_E$ and $T(r(f))\neq E^0$ which is a contradiction. So, there are no sources. Now, suppose that $u\in E^0$ is a sink. Consequently, $\{u\}$ is a hereditary subset different from $E^0$ reaching a contradiction. Now, let $u, v\in E^0$, since $u \in T(v)=E^0$ and $v \in T(u)=E^0$ there exists a path from $u $ to $v$ and other one from $v$ to $u$. If $E^0$ is finite, then it is straightforward  to check that  $E$ is simple if and only if there exists a closed path visiting all the vertices. For item (3), we prove that the set of vertices of the associated graph to $A/I_H$ is $(E/H)^0=E^0\setminus H$ and the set of edges of the associated graph to $A/I_H$ is $(E/H)^1=\{f\in E^1\,\colon\, s(f),r(f)\notin H\}$. Consider a basis of $A$, $B=\{e_i\}_{i\in\Lambda}$ and $B'=\{e_j+I_H\}_{j\in \Gamma}$ a generator system of $A/I_H$ with $\Gamma =\{i\in \Lambda\,\colon\, e_i\notin H\}$. Let us check that $B'$ is a natural basis of $A/I_H$. If we take $\sum_{i\in \Gamma} k_i(e_i+I_H)=0 + I_H$ with $k_i \in \K$, we obtain $\sum_{i\in \Gamma}k_ie_i = \sum_{j \in \Lambda \setminus \Gamma}\rho_j e_j$. Then $k_i=\rho_j=0$ for all $i\in \Gamma$ and $j \in \Lambda \setminus \Gamma$. It is clear that the set of vertices and the set of edges of the associated graph to $A/I_H$ is the same as that of $E / H$.

\end{proof}

Note that the reciprocal of the first statement of the item (2) is not true, it is enough to consider a graph with only two cycles.

\begin{remark}\rm
Observe that, by Proposition \ref{propsimple} and Theorem \ref{teo64}, if $A$ is a finite-dimensional perfect evolution $\K$-algebra, then $A$ is simple if and only if there is a closed path visiting all the vertices
of its associated graph which agrees with \cite[Proposition 2.7]{CKS1}.    
Also, as an illustration of the third item in Theorem \ref{teo64}, consider the three-dimensional evolution algebra $A$ with natural basis $\{e_1,e_2,e_3\}$ and multiplication $e_1^2=e_1+e_2+e_3$, $e_2^2=e_2$
and $e_3^2=e_2+e_3$. This algebra is perfect and its graph $E$ is drawn below. It is easy to realize that there is a unique hereditary subset $H\ne E^0,\emptyset$. In fact, it is $H=\{e_2,e_3\}$. This produces a unique maximal ideal $I_H=\K e_2\oplus \K e_3$. The quotient algebra $A/I_H$ is isomorphic to the ground field and $E/H$ is the graph with a single vertex and a single loop, which corresponds to the one-dimensional evolution algebra $\K$.
\vspace{0.8cm}
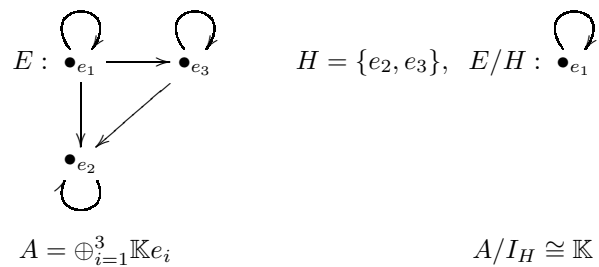
$$
   E:   \xymatrix{
      {\bullet}_{e_{1}}  \ar@(ul,ur) \ar[d] \ar[r] & {\bullet}_{e_{3}} \ar@(ul,ur)  \ar[dl] \\
         {\bullet}_{e_{2}}  \ar@(dr,dl)  &   }
    \hspace{1cm}
    H=\{e_2,e_3\}, \hspace{0.2cm}  E/H:  \xymatrix{
      {\bullet}_{e_{1}}  \ar@(ul,ur)   }
$$
\bigskip
$$A=\oplus_{i=1}^3\K e_i\hspace{4 cm} A/I_H\cong \K$$
\end{remark}

Since algorithms for computing the hereditary subsets of a graph are known (see \cite{atlas}), the task of determining algorithmically the maximal ideals of a given finite-dimensional perfect evolution algebra can be effectively implemented as well as the simplicity of this type of algebras.

\vskip 1cm 
{\bf Data Availability Statement:}  The authors confirm that the data supporting the findings of this study are available within the article.

\end{document}